\documentclass{article}
\usepackage[a4paper,margin=4cm]{geometry}
\PassOptionsToPackage{hyphens}{url}\usepackage{hyperref}
\usepackage{graphicx}

\usepackage{natbib}
\usepackage{amsthm}
\usepackage{amsmath}
\usepackage{mathtools}
\usepackage{enumerate}
\usepackage{amsfonts}
\usepackage{booktabs}
\usepackage{color}
\usepackage{makecell}
\usepackage{authblk}

\newtheorem{theorem}{Theorem}
\newtheorem{observation}{Observation}
\newtheorem{lemma}[theorem]{Lemma}
\newtheorem{corollary}[theorem]{Corollary}
\newtheorem{definition}{Definition}

\def\setaux#1|#2|#3\endsetaux{\if\relax\detokenize{#2}\relax #1 \else #1 \;\delimsize\vert\; #2 \fi}
\DeclarePairedDelimiterX{\set}[1]\lbrace\rbrace{\setaux #1||\endsetaux}

\newcommand{\naturals}{\mathbb{N}}

\title{Finding Fair Draws for Incomplete Round Robin Tournaments}
\author[1]{Sten Wessel\thanks{Corresponding author.}}
\author[1]{Frits Spieksma}
\affil[1]{Department of Mathematics and Computer Science, Eindhoven University of Technology, The Netherlands}
\affil[ ]{\normalsize\texttt{\{s.wessel,f.c.r.spieksma\}@tue.nl}}
\date{}

\begin{document}

    \maketitle

\begin{abstract}\noindent
In an incomplete round robin tournament, teams play against different sets of opponents. Given varying strengths of the teams, this raises a fairness issue. We establish the complexity of determining whether a fair draw exists under various scenario's involving the presence of pots (i.e., strength based groups of teams), teams coming from the same association, connectivity requirements, and others.
We also show experimentally how fair incomplete round robin tournaments can be generated.

    \bigskip


\end{abstract}

    \section{Introduction} \label{sec:intro}
Round robin tournaments are a well-known, well-studied, popular way of organizing a competition.
In such a tournament, every pair of teams (or players) meet each other once (called a Single Round Robin, or SRR), twice (called a Double Round Robin, or DRR), or $\ell$ times ($\ell$RR). A variant of Round Robin tournaments is the so-called \emph{incomplete Round Robin} (iRR).
In this format, it is still the case that every team plays the same number of matches, however, not necessarily against \emph{all} other teams.
Typically, each team plays $k$ matches where $k$ is small compared to the number of participating teams (see Section~\ref{sec:problemdescription} for a precise description).
It follows that teams may face sets of opponents that (wildly) differ in for instance strength.

However, incomplete Round Robins are increasingly popular in practice.
There is the much studied case of the League phase in the UEFA Champions League in football.
This iRR was introduced in the season 2024--2025;
the setting is~$n=36$, $k=8$, where~$n$ refers to the number of teams (see also Section~\ref{sec:literature}).
Another example is the Volleybal Nations League, where, since season 2024--2025, the format is an iRR with $n=18$, $k=12$ \citep[see][]{Lambersetal2023}.
\citet{Devriesere+Goossens2025} describe how the Belgian Youth Hockey league adopted an iRR format.
\citet{Lietal2025} describe the incomplete Round Robin format, and its relation with other formats, in more detail.

There are various reasons for the increasing popularity of an iRR.
UEFA abandoned their traditional format of a setting with $8$~groups of size~$4$ (where each team plays 6 matches prior to the knockout stage) in favor of an iRR;
one of the reasons they give is: ``It will also result in more competitive matches for every club across the board.''\footnote{\url{https://www.uefa.com/uefachampionsleague/news/0268-12157d69ce2d-9f011c70f6fa-1000--new-format-for-champions-league-post-2024-everything-you-ne/}}.
\citet{Csato2025} and \citet{Devriesere+Goossens2025} point out that departing from a regular DRR has the advantage that traditional rankings are harder to construct and interpret, thereby leading to less pressure for teams.
In youth leagues, to keep young players involved, this can be a great advantage.
More generally, while a traditional DRR (or SRR) might be perceived as fair (as each team faces all other opponents) it involves a lot of matches.
iRRs offer the possibility to still rank all teams in relatively few matches.

Of course, a clear issue that needs to be investigated when organizing an iRR is the following.
When not all teams face the same set of opponents, and assuming there is variance in the strength of the teams, the question arises: is the iRR fair?
Indeed, when not all teams have the same strength, is the final ranking a true representation of the strength of the teams?

This paper investigates this question.
We assume that there is given number that represents the strength of each team (or player).
This assumption is a realistic one: almost all sports use ratings to represent the current strength of a team;
from tennis till chess, from football to cricket, ratings are abundantly present.
Also, many organizations that analyze performance in sports use (often Elo-based) ratings to have a reliable estimate of strength.

In an iRR, each team faces a set of $k$~distinct opponents, called an opponent set.
In order to compare the strength of various opponent sets, we use as a measure the average strength of the teams in an opponent set.
Of course, it is true that opponent sets with the same average strength might still be perceived differently: one opponent set may consist of teams with almost equal strength, while another opponent set with the same average strength may consist of teams whose strength vary quite a bit.
However, we find the average a natural choice to represent the strength of an opponent set.

A collection of $n$ opponent sets such that (i) there is one opponent set associated to each team, and (ii) if team $i$ is in the opponent set of team $j$, then team $j$ is in the opponent set of team $i$, is referred to as a \emph{draw}.
Our measure of the fairness of a draw is the \emph{bandwidth}: the difference between the strongest opponent set and the weakest opponent set.
In case that difference equals $0$, i.e., if the strengths of all $n$ opponent sets in a draw are equal, we call the draw \emph{fair}.
A precise problem description is given in Section~\ref{sec:problemdescription}.

We present two types of results.
From a theoretical perspective we establish the complexity of finding fair draws.
More specifically:
\begin{itemize}
    \item We prove that finding a fair draw is solvable in polynomial time for $k=2$.
    This also holds when additional requirements are present, notably: a number of pots ($p$), and the presence of forbidden matches induced by teams belonging to the same association, where we use~$c$ to denote the number of associations.
    \item We show that finding a fair draw is NP-complete for $k=3$.
    \item We show that deciding whether a so-called connected draw (see Section~\ref{sec:problemdescription}) exists is NP-complete, even for $k=2$.
\end{itemize}

These complexity results are tabulated in Tables~\ref{table:overview2} and \ref{table:overview3} for $k=2$ and $k= 3$, respectively.
\begin{table}[h]
    \centering
    \renewcommand{\arraystretch}{1.2}
    \setlength{\tabcolsep}{8pt}
    \caption{Overview of complexity results for $k=2$.}\label{table:overview2}
    \begin{tabular}{c|cc}
        \toprule
        $k=2$ & $c=n$ & $c=2$ \\
        \midrule
        $p=1$ & \makecell{P\\(Theorem~\ref{thm:stengths-2})} & \makecell{P\\(Theorem~\ref{thm:k2-iRR2-1})}  \\[3ex]
        $p=k$ & \makecell{P\\(Theorem~\ref{thm:k2-iRRn-2})} &  \makecell{P\\(Corollary~\ref{thm:k2-iRR2-2})} \\
        \bottomrule
    \end{tabular}
\end{table}
\begin{table}[h]
    \centering
    \renewcommand{\arraystretch}{1.2}
    \setlength{\tabcolsep}{8pt}
    \caption{Overview of complexity results for $k=3$.}\label{table:overview3}
    \begin{tabular}{c|cc}
        \toprule
        $k=3$  & $c=n$ & $c=2$ \\
        \midrule
        $p=1$ & \makecell{NP-complete\\(Theorem~\ref{th:basecase})} & \makecell{NP-complete\\(Theorem~\ref{th:iRR2-1})} \\[3ex]
        $p=k$ & \makecell{NP-complete\\(Theorem~\ref{th:iRRn-3})} & \makecell{NP-complete\\(Theorem~\ref{th:iRR2-3})}  \\
        \bottomrule
    \end{tabular}
\end{table}

 We also carry out experimental investigations.
 We generate instances with randomly generated strengths.
 Our goal here is to study properties of random draws, as well as fairest possible draws.
 The latter ones we find using integer programming.
 We compare the bandwidth, and the connectivity of all these instances.
 We observe the following:
 \begin{itemize}
     \item The minimum bandwidth is an order of magnitude smaller than bandwidth of a random draw.
     \item The minimum-bandwidth draws are only slightly less connected than the uniformly-sampled random draws.
 \end{itemize}

Our paper is organized as follows. Section~\ref{sec:problemdescription} gives a detailed problem description.
In Section~\ref{sec:k=2}, we give characterizations and complexity results for draws where teams play against two opponents.
We prove our complexity results for the case where teams have at least three opponents in Section~\ref{sec:kgeq3}.
In Section~\ref{sec:compexp}, we discuss our experimental findings.
We conclude and discuss implications for real-world applications in Section~\ref{sec:conclusion}.

\subsection{Literature} \label{sec:literature}
As far as we are aware, \citet{Froncek2013} first described the problem of finding a fair incomplete round robin tournament.
He ranks the teams from $1$ to $n$, and uses the rank as a measure of the strength of a team.
In that case, \citet{Froncek2013} shows that a finding a fair schedule is equivalent to finding a so-called distance magic labeling of any $k$-regular graph on $n$ vertices, see \citet{Froncek+Shepanik2016}.
This line of research is continued in \citet{FreKer2023}.

One prominent example of an iRR is the UEFA Champions League.
Various properties of this particular iRR were analyzed in \citet{Csatoetal2025, Guyonetal2025, Devriesereetal2025}.

Fairness has been addressed quite a bit for various fairness measures other than balancing the strengths of opponent sets in incomplete round robin draws.
\citet{Melkonian2024} describe an integer programming technique to schedule incomplete round robin tournaments using a fairness objective that is different from this work.
\citet{Tan+Miy2016} considers fairness in home-away patterns of matches in an actual competition schedule, rather than draws.
This also relates to the minimization of the number of home and away breaks present in the timetable, as described by \citet{DeW+Urr+Ass2026}, among others.

In our work, we assume that the ranking of the teams is based on the total number of points collected by each team, and does not depend on the strength of the particular teams played against.
This is in line with the practice used in the examples used in Section~\ref{sec:intro}.
Another way of dealing with opponent sets of varying strengths, is to let the ranking depend on the strength of the opponent set of each team.
This is a research theme in itself, we refer to \citet{rankingmatters}, \citet{betran}, \citet{keener}, and \citet{lambers2020true}, and the references contained therein.

\section{Problem Description}\label{sec:problemdescription}
We use $T$ to denote the set of teams, with $n=|T|$.
We let $k \leq n-1$ equal the number of matches each team must play.
In case $k=n-1$, every team plays every other team exactly once;
this corresponds to the well-studied Single Round Robin.
While we see this as a special case, we are primarily interested in settings where $k$ is much smaller than $n$.

Given $T$, and $k$, our objective is to find a so-called \emph{draw} defined as follows.

\begin{definition}
    Given a set of teams $T$ and an integer $k, 1 \leq k \leq |T|-1$, we denote by a \emph{draw} a set of $\frac{nk}{2}$ matches (i.e., unordered pairs of teams from $T$) such that (i) each team occurs in $k$ matches, and (ii) no match occurs more than once.
\end{definition}

We observe that a draw differs from a timetable; in a timetable matches are typically assigned to rounds such that each team plays once in each round.
To point out the difference between a draw and a timetable, consider, as an example, a setting with $T=\{A, B, C\}$ and $k=2$.
While there is a draw for this instance (consisting of the three matches $A$ vs $B$, $A$ vs $C$, and $B$ vs $C$), no timetable exists.
Indeed, there is no way to schedule the three matches such that in each round each team plays once.
A consequence of our interest in draws as opposed to timetables is that fairness issues that come from home/away considerations, or breaks \citep[see][]{DeW+Urr+Ass2026} are not within our scope.

In a draw, each team occurs in $k$ matches, and hence faces a set of $k$ distinct opponents.
We call such a set of teams an \emph{opponent set}.
Thus, an opponent set is a set of $k$ teams that can act as the set of opponents of a team in the draw.
For instance, the opponent set of team A in the example above is $\{B,C\}$

We use $s_i$ to denote the strength (or rating) of team $i \in T$.
An indication of strength is abundantly present in the world of sports, see Section~\ref{sec:intro};
for instance, the $s_i$s can represent Elo ratings.
In order to compare the strength of various opponent sets, we use the average strength of the teams in the opponent set.

We are interested in the following question: what draw gives rise to the smallest difference between the highest strength of an opponent set in the draw and the lowest strength of an opponent set in the draw?
We call this objective the \emph{bandwidth}.

We now specify the corresponding problem:

\begin{description}
    \item \textsc{Bandwidth-}i\textsc{RR} ($\underline{\overline{iRR}}$): given $k$, and given $n$ teams with their strengths $s_i$, $i \in T$, find a draw that minimizes the bandwidth.
\end{description}

Clearly, $\underline{\overline{iRR}}$ is an optimization problem.
We analyze this problem experimentally in Section~\ref{sec:compexp}.
We are also interested in the question whether a bandwidth of value $0$ is possible.
As mentioned in Section~\ref{sec:intro}, we refer to a draw with bandwidth $0$ as a \emph{fair} draw.
This leads to the following decision problem:

\begin{description}
    \item \textsc{Decision-}i\textsc{RR} ($iRR$): iven $k$, and given $n$ teams with their strengths $s_i$, $i \in T$, does there exist a fair draw, i.e., does there exist a draw such that the strengths of all opponent sets are equal?
\end{description}

Motivated by practice, we consider two additional parameters that can be part of the input of an iRR\@.
\begin{itemize}
    \item It may well be that not every pair of teams is allowed to play a match in the iRR\@.
    For instance, in the UEFA Champions League teams from the same (national) league do not meet in the iRR\@.
    As another example, in amateur settings, it makes sense to ensure that teams from the same club do not meet in the iRR\@.
    To be able to take this into account, we introduce a parameter $c$ that stands for the number of nonempty sets in a partition of the set of teams $T$;
    informally, one can think of $c$ as the number of ``national leagues'' or ``associations'' that the participating teams originate from.
    Teams from the same set (or associations) are not allowed to play against each other in the iRR\@.
    Notice that if $c=n$ (where $n$ is the number of teams), all matches are allowed.
    \item A tournament organizer may distribute the teams over (strength-based) \emph{pots}.
    The idea is that each team plays matches against teams from each pot, so that some amount of fairness is realized.
    Let $p$ denote the number of pots.
    We will assume that $n$ is a multiple of $p$, and each pot contains $\frac{n}{p}$ teams.
    We also assume that $k$ is a multiple of $p$, and each team has to play $\frac{k}{p}$ matches against teams from each pot.
    Notice that if $p=1$, any set of $k$ teams qualifies as an opponent set.
\end{itemize}
We will use the phrase \emph{base case} to refer to $iRR$, i.e., the setting where $c=n, p=1$.
Cases with other values for $c$ or $p$ are collectively referred to as \emph{extended cases}, and are denoted by $iRR_{c,p}$.
Of course, we could refer to the base case $iRR$ as $iRR_{n,1}$;
for reasons of convenience, we simply use $iRR$ for the base case.

A final interesting aspect of the draw is its connectivity.
For this, consider the following graph that naturally corresponds to a draw.
\begin{definition}[Draw graph]
    Given a draw, the \emph{draw graph} is a graph~$H$ that has a vertex for each team, and two vertices are connected by an edge if and only if the two corresponding teams play each other.
\end{definition}
Clearly, in the draw graph each vertex has degree $k$, i.e., $H$ is $k$-regular.
However, it may well happen that this graph $H$ is disconnected.
In fact, subgraphs consisting of only a few vertices (in case $k=2$: triangles, $C_4$s) may well be part of the solution.
We argue that small disconnected subgraphs are not attractive from a practical point of view.
The presence of a relatively large set of matches played among few teams is opposite to the idea of an iRR where one wants to avoid ``subleagues'' that would generate their own standings, perhaps even styles of play, and which do not, or not often, play outside their ``clique''.
This leads to the following definition.
\begin{definition}[Connected draw]
    A \emph{connected draw} is a draw whose draw graph is connected.
\end{definition}
If a draw is connected, we furthermore quantify the connectedness by reporting (i) the diameter of $H$ (i.e., the length of a longest shortest path between any two vertices in $H$), and (ii) the average length of all shortest paths between any pair of vertices in $H$.

\section{The complexity of $iRR$ with two opponents}\label{sec:k=2}
In this section we show that $iRR$ is solvable in polynomial time for $k=2$.
In Section~\ref{sec:basecasek=2} we provide an algorithm for $iRR$ that produces a draw with equal strength for each opponent set in the draw if such a draw exists.
In Section~\ref{sec:extcasek=2} we show how to deal with cases where there are two pots ($p=2$), and/or two associations ($c=2$).
In Section~\ref{sec:connectivity}, we show that, when given arbitrary admissible opponent sets, it is NP-complete to decide whether a connected draw exists, even for $k=2$.

\subsection{The base case}\label{sec:basecasek=2}
Recall that we are given $n$ teams with strengths~$s_i \in \naturals$ for every~$i \in T$.
The \emph{opponent sets} are the pairs~$\set{i, j}$ such that~$s_i + s_j = \frac{2 \sum s_i}{n}$.
Let~$\bar s = \sum s_i/n$ denote the average strength.
Consider now the following graph, which we call the \emph{opponent graph}.
\begin{definition}[Opponent graph]
    The \emph{opponent graph} is the graph~$G=(V, E)$ where the vertices~$V = [n]$ represent the teams and the edges~$E = \set{\set{i, j} : \frac 1 2 (s_i + s_j) = \bar s}$ are the admissible opponent sets.
\end{definition}
From this definition, the following observation is straightforward, as every team needs to appear in exactly two (not necessarily distinct) opponent sets.
\begin{observation} \label{obs:p2m}
    If a fair draw exists, then the opponent sets used in the draw form a perfect $2$-matching in~$G$.
\end{observation}
Therefore, $G$ must admit a perfect $2$-matching for a draw to exist. However, this is not a sufficient condition.
Take for example an opponent graph that is isomorphic to~$C_6$, i.e., a cycle on six vertices.
Clearly, this graph has a perfect $2$-matching, but Lemma~\ref{lem:bad-cycle} shows that no fair draw exists.
We now provide five lemmas providing necessary and sufficient conditions for a draw to exist.

\begin{lemma} \label{lem:g-structure}
    If a draw exists, then~$G$ is isomorphic to the graph~$K_{k_0} + \sum_{d \in D} K_{k_d,k_d}$, where the set~$D$ is defined as~$D  = \set{d \in \naturals_{>0} : E_d \neq \emptyset}$, $k_0 \ge 0$, $k_0 \neq 1$, $k_d \ge 1$ for all~$d \in D$.
\end{lemma}
\begin{proof}
    The edges of~$G$ can be partitioned into sets~$E_d = \set{\set{i, j} : s_i = \bar s - d,\ s_j = \bar s + d}$, for all~$d \ge 0$.
    Note that~$G[E_0]$ forms a clique, where for~$d > 0$ the induced subgraph~$G[E_d]$ is a complete bipartite graph with vertex sets~$V^-_d = \set{i : s_i = \bar s - d}$ and~$V^+_d = \set{i : s_i = \bar s + d}$.
    It is a fact that a perfect 2-matching in $G$ exists if and only if there is a perfect 2-matching in every connected component of the graph.
    The clique always admits a perfect 2-matching if it is not an isolated vertex.
    Every bipartite component admits a perfect 2-matching if and only if $|V^-_d| = |V^+_d|$.
\end{proof}

\begin{lemma} \label{lem:useful}
    Let $C$ be a cycle in a perfect 2-matching of the opponent graph with even length $r$.
    If there is a feasible draw that is consistent with the perfect 2-matching, then the teams in $C$ must play exactly the set of teams of another cycle of the same length $r$ in the perfect 2-matching.
\end{lemma}
\begin{proof}
    Let $1, \dots, r$ be the teams on the cycle $C$.
    First, we show that the teams in the cycle cannot only play each other.
    Consider an arbitrary draw where, without loss of generality, team 1 is assigned to opponent set $\{k, k+1\}$ for some $k$. Consider team 2. There are two options: either it is assigned to $\{k-1,k\}$, or it is assigned to $\{k+1,k+2\}$. All subsequent assignments follow from this choice. In fact, the first option leads to a contradiction as it directly follows from this choice that team $i$ is assigned to opponent set $\{k-i+1, k-i+2\}$ for $i=1, \ldots, k$. When choosing $i=\frac{k}{2}+1$ in case $k$ is even or $i=\frac{k}{2} + \frac12$ in case $k$ is odd, the team is present in its own opponent set.

    Thus, we can assume that if a solution exists, it is one where team 1 is assigned to opponent set $\{k, k+1\}$, and team 2 is assigned to $\{k+1,k+2\}$. Clearly, that means that team $k+1$ is assigned to opponent set $\{1,2\}$. But it also implies that team $i$ is assigned to opponent set $\{k+i-1, k+i\}$ for $i=1,...,k$. When plugging in $i=k+1$, we find that the opponent set of team $k+1$ is $\{2k,2k+1\}$, leading to $2k=r+1$, which is a contradiction as the left hand side is even, and the right hand side is odd.

    Hence, one of the teams, w.l.o.g.\ team 1, plays a team outside of the cycle $C$, say opponent set $\{t_1, t_2\}$, where $t_1$ and $t_2$ are both not in $C$.
    Then, $t_1$ and $t_2$ must be assigned to the opponent sets $\{1, 2\}$ and $\{1, r\}$, and by the same argument as above, teams $2$ and $r$ must play opponent sets containing $t_1$ and $t_2$.
    This yields that the opponent sets of the teams $1, \dots, r$ must form a cycle of length $r$.
\end{proof}

\begin{lemma} \label{lem:bad-cycle}
    If~$G = C_{4m + 2}$, $m \ge 0$, then a fair draw does not exist.
\end{lemma}
\begin{proof}
    We assume that the teams in $C_{4m+2}$ are indexed in clockwise order $1, 2, \ldots, n$. First, we observe that no $C_2$ can be present in a perfect 2-matching for $C_{4m+2}$. Indeed, let $C_2=\{i,i+1\}$ for some $1 \leq i \leq n$ be present in a perfect 2-matching, and consider team $i-1$. This team must be present in two opponent sets, and as the opponent set $\{(i-1,i)\}$ is no longer an option (due to using $(i,i+1)$ twice), we must also use $C_2=\{i-2, i-1\}$ in the perfect 2-matching. Repeating this argument implies that if the solution contains one $C_2$, it consists of only $C_2$s. But that is impossible as the number of $C_2$s would be odd (given that $n=4m+2$).

    Second, we observe that this means that the opponent sets must form the complete cycle $G$, and by Lemma~\ref{lem:useful}, a fair draw can thus not exist.
\end{proof}

\begin{lemma} \label{lem:bipartite-even}
    If~$G = K_{m,m}$, $m \ge 2$ and $m$ is even, then a fair draw exists.
\end{lemma}
\begin{proof}
    Take any perfect matching in~$G$, which consists of $m$ edges; an even number.
    Partition the matching edges into disjoint pairs.
    For every such pair~$\set{i, j}, \set{k, \ell}$, let $i$ and $j$ play against~$\set{k, \ell}$, and~$k, \ell$ play against~$\set{i, j}$.
\end{proof}

\begin{lemma} \label{lem:bipartite-odd}
    Let~$k$ be odd, and let~$m_i \ge 0$, $m_i$ odd for all $i \in [k]$.
    If~$G = K_{m_1,m_1} + \dots + K_{m_k,m_k}$, then a draw does not exist.
\end{lemma}
\begin{proof}
    We will show that no fair draw exists by showing that there does not exist a feasible assignment of teams to opponent sets for every perfect 2-matching in~$G$.
    Consider any perfect 2-matching in~$G$.
    In any connected component, this perfect 2-matching consists of disjoint even cycles (where~$C_2$, a single edge, is considered an even cycle of length~$2$).
    We partition these cycles into sets~$A$ and~$B$, containing cycles of length~$0 \mod 4$ and~$2 \mod 4$, respectively.
    Note that the total number of teams in~$G$ is~$2 \sum_{i=1}^k m_i \equiv 2 \mod 4$.
    Since the cycles in~$A$ and~$B$ cover all teams, $|B|$ must be odd.

    By Lemma~\ref{lem:bad-cycle}, teams that are in a cycle in~$B$ cannot have opponent sets that consist of only teams of that cycle.
    Consider a cycle~$c_1$ in~$B$.
    Thus, at least one team on this cycle, say~$i$, must play an opponent set with both opponents on a different cycle~$c_2$, say~$\set{j, \ell}$.
    Now, team~$j$ and~$\ell$ must play against~$i$, and therefore must be assigned the two opponent sets incident with~$i$ on cycle~$c_1$.
    The neighbors of~$i$ on the cycle must play~$j$ or~$\ell$, repeating the argument to their neighbors.
    Hence, all teams of~$c_1$ must be matched with (all) opponent sets of~$c_2$ and vice versa for a feasible draw to exist using the opponent sets of this perfect 2-matching.
    Therefore, $c_1$ and~$c_2$ need to be of equal length, and in particular both cycles belong to the set~$B$.
    Since the number of cycles in~$B$ is odd, we cannot match all cycles with another of equal length: a contradiction.
\end{proof}

We are now ready to prove the main result of this section.
\begin{theorem}\label{thm:stengths-2}
    There exists a fair draw precisely when the opponent graph is of the structure in Lemma~\ref{lem:g-structure} and:
    \begin{itemize}
        \item $k_0 = 0$ and $|\set{k_d : \text{$d \in D$, $k_d$ is odd}}|$ is even, or
        \item $k_0 \in \{2,3\}$ and $|\set{k_d : \text{$d \in D$, $k_d$ is odd}}|$ is odd, or
        \item $k_0 = 4$ and $|\set{k_d : \text{$d \in D$, $k_d$ is odd}}|$ is even, or
        \item $k_0 \ge 5$.
    \end{itemize}
\end{theorem}
\begin{proof}
    Let $n_o \coloneqq |\set{k_d : \text{$d \in D$, $k_d$ is odd}}|$ be the number of odd complete bipartite graphs.
    We consider cases depending on the size of the clique and the number of odd bipartite graphs separately:
    \begin{enumerate}[(i) ]
        \item $k_0 = 0$, $n_o$ odd.
        A draw does not exist, by Lemma~\ref{lem:bipartite-odd}.
        \item $k_0 = 0$, $n_o$ even.
        A draw exists, by applying Lemma~\ref{lem:bipartite-even} to every even bipartite component.
        \item $k_0 = 1$.
        A draw does not exist, as a perfect $2$-matching does not exist.
        \item $k_0 = 2$, $n_o$ odd.
        When considering~$K_2$ as a~$K_{1,1}$, this instance is equivalent with case (ii) and thus a draw exists.
        \item $k_0 = 2$, $n_o$ even.
        When considering~$K_2$ as a~$K_{1,1}$, this instance is equivalent with case (i) and thus a draw does not exist.
        \item $k_0 = 3$, $n_o$ odd.
        The teams in the clique must only play each other and the remaining instance is equivalent with case (i).
        Hence, a draw does not exist.
        \item $k_0 = 3$, $n_o$ even.
        The teams in the clique must only play each other and the remaining instance is equivalent with case (ii).
        Hence, a draw exists.
        \item $k_0 = 4$, $n_o$ even.
        The clique can be resolved internally (as it contains~$K_{2,2}$ as a subgraph) and this instance is thus equivalent with case (ii).
        Therefore a draw exists.
        \item $k_0 = 4$, $n_o$ odd.
        Consider any perfect 2-matching in the clique.
        This is either a matching or a cycle of length 4.
        When it is a matching, this instance is equivalent to case (i), and thus a draw does not exist.
        When the perfect 2-matching is a cycle, it can only be resolved when it is matched with another cycle of length~4, which must come from a perfect 2-matching in an odd bipartite component.
        This does not yield a feasible draw by Lemma~\ref{lem:useful}.
        \item $k_0 \ge 5$, $n_o$ odd.
        Decompose the odd bipartite graphs in matching edges, of which there are an odd number.
        Take one edge from the clique; then a schedule on the matching edges and this additional edge can be formed by pairing edges that play against each other.
        Then, the clique that remains is of size at least~$3$.
        The remaining schedule consists of teams in the remaining clique playing each other and applying Lemma~\ref{lem:bipartite-even} on the remaining even bipartite components.
        \item $k_0 \ge 5$, $n_o$ even.
        When the clique has an odd size, the teams on the clique can play each other.
        When the clique is of even size, it can be decomposed into two complete bipartite graphs that both have~$k_0/2$ vertices, which can play each other.
        The remaining instance is equivalent to case~(ii), and thus a draw exists.
        \qedhere
    \end{enumerate}
\end{proof}

Theorem~\ref{thm:stengths-2} not only characterizes instances where a fair draw exists, but also sketches an algorithmic approach to recognize these instances in polynomial time.
In fact, for any given instance, checking whether a fair draw exists can be done in linear time, as it is sufficient to count how many teams have a certain strength.
Indeed, a perfect matching in the bipartite components of the opponent graph exists if and only if the number of teams of strength~$\bar s - d$ and $\bar s + d$, for every $d \ge 1$, are equal.
Then, the conditions of Theorem~\ref{thm:stengths-2} can be checked by counting the number of teams with strength~$\bar s$.
This leads to the following corollary.
\begin{corollary}
    $iRR$ for $k=2$ can be solved in linear time.
\end{corollary}

\subsection{The extended case} \label{sec:extcasek=2}
In this section, we extend the base case to situations where either~$c=2$ and/or~$p=2$.

We first consider the case where the number of associations is~$2$, i.e., we consider $iRR_{2,1}$.
Now, the vertices in the opponent graph can be considered to be colored either blue or red, corresponding to the association the team belongs to.
First, note that the opponent graph is a subgraph of the graph obtained in the base case: only edges between vertices of the same color remain.
Let~$R$ and~$B$ define the teams colored red and blue, respectively, and define~$G_c$ as the colored opponent subgraph of $G$ with edges~$E_c = E \cap \set{\set{i, j} : \text{$i,j \in R$ or $i,j\in B$}}$.
Thus, the clique in the opponent graph will split into a red and blue clique, and each bipartite component will be split into red and blue complete bipartite graphs.
Let~$k_0^r$ and~$k_0^b$ denote the number of vertices in the red and blue clique, respectively, and let~$k^r_\mathrm{bp}$ and~$k^b_\mathrm{bp}$ denote the number of red and blue vertices in the bipartite components of~$G_c$, respectively.
Further, let $k^r = k_0^r + k^r_\mathrm{bp}$ and $k^b = k_0^b + k^b_\mathrm{bp}$
Without loss of generality, we will assume that~$k^r_\mathrm{bp} \ge k^b_\mathrm{bp}$.

\begin{theorem}\label{thm:k2-iRR2-1}
    There exists a fair draw for $iRR_{2,1}$ for $k=2$ precisely when $k^r = k^b$ and each of the bipartite components have a perfect matching.
\end{theorem}
\begin{proof}
    For the one direction, assume that a fair draw exists.
    In a feasible draw, a red team must play a pair of blue teams, and vice versa, and thus $k^r = k^b$.
    Then the opponent sets of the fair draw correspond to a perfect 2-matching $G_c$.
    The perfect 2-matching consists of separate matching edges and even cycles in the bipartite components of $G_c$.
    The even cycles can be decomposed into a matching, yielding a perfect matching in the bipartite components.

    For the other direction, let each of the bipartite components have a perfect matching and let $k^r = k^b$.
    We construct a draw as follows.
    We first assign the teams in the bipartite components, using the edges of the perfect matching.
    For every blue edge~$\{b_1, b_2\}$, we use a unique red edge~$\{r_1, r_2\}$ and we assign $b_1$ and $b_2$ the opponent set $\{r_1, r_2\}$ while we assign $r_1$ and $r_2$ the opponent set~$\{b_1, b_2\}$.
    For the remaining $k^r_\mathrm{bp} - k^b_{\mathrm{bp}}$ red teams, we will take the same amount of blue teams from the clique.
    We form a perfect matching on these blue teams and we assign opponent sets analogously as above.
    What remains are $k_0^b - (k^r_\mathrm{bp} - k^b_{\mathrm{bp}}) = k_0^r$ blue and red teams in both cliques.
    As opponent sets, we use a two cycles in the red and blue cliques, and assign opponents sets according to the proof of Lemma~\ref{lem:useful}.

    In a feasible draw, a pair of blue teams will serve as an opponent set for a red team, and vice versa.
    The opponent sets of the draw form a perfect 2-matching in~$G_c$, which exists if and only if the bipartite components have a perfect matching.
    In the bipartite components of~$G_c$, a perfect 2-matching consists of even cycles and matching edges.
    Therefore, without loss of generality, we can assume the perfect 2-matching only consists of matching edges in the bipartite components.
    Every blue edge must be matched with a red edge.
    If~$k^r_\mathrm{bp} > k^b_\mathrm{bp}$, then~$k^r_\mathrm{bp} - k^b_{\mathrm{bp}}$ additional red edges may be obtained from the red clique.
    The remaining teams in the cliques must play each other.
    Thus, the $k_0^r - (k^r_\mathrm{bp} - k^b_{\mathrm{bp}})$ remaining number of red and~$k_0^b$ blue teams must be equal, which can be matched using two equal-length cycles of length~$k_0^b$.
\end{proof}

We now turn to the extended case with two pots, i.e., $iRR_{n,2}$
The vertices in the opponent graph can also be considered to be colored either blue or red, corresponding to the pot it is contained in.
Let~$R$ and~$B$ define the teams colored red and blue, respectively, and define~$G_p$ as the colored opponent subgraph of $G$ with edges~$E_p = E \cap \set{\set{i, j} : \text{$i \in R,j \in B$ or $i\in B, j \in R$}}$.
Thus~$G_p$ is a subgraph of~$G$ where only the bichromatic edges remain.
Every component of~$G_p$ is therefore bipartite.

\begin{theorem}\label{thm:k2-iRRn-2}
    There exists a fair draw for $iRR_{n,2}$ for $k=2$ precisely when every component of~$G_p$ has a perfect matching and the number of teams is a multiple of~$4$.
\end{theorem}
\begin{proof}
    The opponent sets of the draw form a perfect 2-matching in~$G_p$, which exists if and only if the components---which are all bipartite---have a perfect matching.
    Therefore, every perfect 2-matching consists of even cycles and matching edges.
    Therefore, without loss of generality, we can assume the perfect 2-matching only consists of matching edges.
    Since matching edges need to be paired to play each other, the number of matching edges must be even, and thus the number of teams is a multiple of~$4$.
\end{proof}

Lastly, we consider $iRR_{2,2}$.
Notice that the opponent graph, which we call $G_{c,p}$ is a subgraph of $G$ with edge set~$E_{c,p} = E_c \cap E_p$, i.e., only opponent sets remain that contain two teams from the same association, but from two different pots.
It can be directly verified that, since $G_{c,p}$ is a subgraph of $G_p$, the statement and proof of Theorem~\ref{thm:k2-iRRn-2} hold for $iRR_{2,2}$ when applied to the opponent graph~$G_{c,p}$.
We thus get the following result.

\begin{corollary}\label{thm:k2-iRR2-2}
    There exists a fair draw for $iRR_{2,2}$ for $k=2$ precisely when every component of~$G_{c,p}$ has a perfect matching and the number of teams is a multiple of~$4$.
\end{corollary}

\subsection{Connectivity}\label{sec:connectivity}
We give a characterization for the existence of a connected draw in Section~\ref{sec:conn:str}. We show that finding a connected draw is NP-complete for the setting where the collection of admissible opponent sets is represented by an arbitrary graph in Section~\ref{sec:conn:g}.

\subsubsection{When teams have given strengths} \label{sec:conn:str}
In this section, we assume that the strengths of the teams are given.
As described above (see Lemma~\ref{lem:g-structure}), the opponent graph consists of a single clique together with complete bipartite components.
\begin{theorem}
    There exists a fair connected draw for $iRR$ for $k=2$, precisely when a draw exists and either:
    \begin{enumerate}[(i) ]
        \item all teams have the same strength,
        \item $n/2$ teams have strength $s - d$ and $n/2$ teams have strength $s + d$, for some $s,d \ge 0$,
        \item $n/2$ teams have strength~$s$ for some~$s \ge 0$, $n/4$ teams have strength $s-d$ and $n/4$ teams have strength~$s+d$, for some~$d > 0$,
        \item $n/4$ teams have strengths~$s-d, s + d, s-e, s+e$, respectively, for some $s \ge 0$, $d,e \ge 0$, $d \neq s$.
    \end{enumerate}
\end{theorem}
\begin{proof}
    Note that in order to have a connected draw, the opponent graph can have at most two components, and therefore the teams can have at most four distinct strengths by Lemma~\ref{lem:useful}.
    Assume that a feasible draw exists.
    We show that exactly when one of the conditions (i)--(iv) holds, there exists a connected draw.
    We consider the number of distinct teams strengths separately.
    \begin{enumerate}[(i) ]
        \item When all teams have the same strength, the opponent graph is a clique. When the number of teams is odd, a connected draw forms a single cycle through all teams in the opponent graph.
        When the number of teams is even, a connected draw exists where the opponent sets form two equal-sized cycles, with the draw structured according to Lemma~\ref{lem:useful}.
        \item When teams have two distinct strengths, the opponent graph is a bipartite graph with one connected component.
        Then, a connected draw exists where the opponent sets form two equal-sized cycles, with the draw structured according to Lemma~\ref{lem:useful}.
        \item  When teams have three distinct strengths,
        if a feasible draw exists, then the strengths must be selected from~$\set{s, s-d, s+d}$ for some~$s \ge 0$ and $d > 0$.
        A connected draw only exists when the size of the clique is equal to the size of the bipartite component in the opponent graph, hence $n/2$ teams have strength~$s$ and $n/4$ teams have strength~$s-d,s+d$, respectively.
        \item When teams have four distinct strengths,
        if a feasible draw exists where the opponent graph has at most two components, then the strengths must be selected from~$\set{s-d, s+d, s-e, s+e}$ for some~$s \ge 0$ and $d,e > 0$, $d \neq e$.
        Since the two components need to play each other, each strength must be assigned to $n/4$ teams.\qedhere
    \end{enumerate}
\end{proof}

\subsubsection{Arbitrary opponent sets} \label{sec:conn:g}
In this section, we assume that the collection of admissible opponent sets is represented by a graph $G$, where each vertex corresponds to a team, and each edge between a pair of teams represents an admissible opponent set.
Clearly, this is more general compared to the case where strengths $s_i$ are given.

As described in Section~\ref{sec:problemdescription}, a draw can be represented by a graph $H$ that has a vertex for each team, and where two teams are connected if they play against each other according to the draw.
When given an arbitrary graph $G$ describing the admissible opponent sets, it is NP-complete to decide whether a connected draw exists, as witnessed by the following theorem.

\begin{theorem}
    Connected $iRR$ (given $G$, does there exist a connected draw) is NP-complete, even for $k=2$.
\end{theorem}
\begin{proof}
    We reduce from Hamiltonicity: given a graph $G'$, is it Hamiltonian? We set $G\coloneqq G'$ where a vertex represents a team, and an edge represents an admissible opponent set. Let the teams be indexed $1, 2, \dots, n$, and let us assume, wlog, that $n=4k+1$. We now prove the equivalence between existence of a Hamiltonian cycle in $G'$, and the existence of a connected draw.

    If $G'$ is Hamiltonian, reindex the teams according to their sequence in the Hamiltonian cycle, starting arbitrarily at some vertex that we index with $1$.
    Then, we denote the edges of the Hamiltonian cycle by $(i,i+1)$ for $i=1, \dots, n-1$ and $(n,1)$.
    To build a draw, we assign team $1$ to edge $(\lceil \frac{n}{2} \rceil, \lfloor \frac{n}{2} \rfloor)$, and assign teams $1+i$ to edges $(\lceil \frac{n}{2} \rceil +i, \lfloor \frac{n}{2} \rfloor+i)$ for $i=1, \ldots, n-1$ indices modulo $n$.
    Notice that this draw uses only edges from $G'$, and is connected by construction.

    For the other direction, suppose there exists a connected draw, i.e., the graph $H$ is connected.
    Recall that the graph $H$ is 2-regular, which is a collection of cycles.
    As $H$ is connected, there is actually one cycle containing all vertices; let us index these vertices along this cycle: $1, 2, \ldots, n=4k+1$.
    We observe that team $i$ plays against its neighbors $i-1$ and $i+1$ for $i=1, \ldots, 4k+1$ (indices modulo $4k+1$).
    It follows that the edges in $(i-1, i+1)$ are in the edge set of $G'$.
    That leads to the conclusion that $G'$ is Hamiltonian, as the vertices $1, 3, 5, \ldots, 4k+1, 2, 4, \ldots, 4k, 1$ form a Hamiltonian cycle in~$G'$.
\end{proof}

\section{The complexity of $iRR$ with at least three opponents}\label{sec:kgeq3}
In Section~\ref{sec:basecasekgeq3} we prove that $iRR$ becomes NP-complete when $k= 3$.
In Section~\ref{sec:extendedcasekgeq3} we show similar results in case there are multiple pots, and/or multiple associations.
In our reductions, we use Numerical 3-Dimensional Matching (N3DM), which is specified as follows.

\begin{description}
    \item \textsc{Numerical 3-Dimensional Matching} (N3DM)
    \item \emph{Input:} Three $m$-sets of positive integers $\{x_i\}$, $\{y_j\}$, and $\{z_k\}$, and an integer $B$, with $0 < x_i,y_i,z_i < B$, for all~$i = 1, \dots, m$.
    \item \emph{Question:} Do there exist $m$ triples of the form $(x_i, y_j, z_k)$ whose three elements sum to $B$ such that each integer in the input is in one triple?
\end{description}

\subsection{The base case}\label{sec:basecasekgeq3}
\begin{theorem}\label{th:basecase}
    $iRR$ is NP-complete for $k= 3$.
\end{theorem}
\begin{proof}
     Clearly, $iRR$ is in the class NP. Given an instance of N3DM, we construct an instance of $iRR$ as follows.
    Fix a sufficiently large constant~$L \ge 3B$.
    Let $k\coloneqq3$, and let there be~$n \coloneqq 6m$ teams.
    We consider three types of teams, each type consisting of~$2m$ teams.
    There are~$m$ teams of type~(i) with strengths~$\{x_i\}$, $m$ teams of type~(ii) with strengths $\{y_j + L\}$, and $m$ teams of type~(iii) with strengths $\{z_k + L^2\}$; in addition there are $m$ teams of type~(i) with strength~$B$, $m$ teams of type~(ii) with strength $L$, and $m$ teams of type~(iii) with strength $L^2$.
     These latter $3m$ teams are referred to as dummy teams.
    The question is whether there exists a fair draw where every opponent set has total strength~$B + L + L^2$.

    We now prove the equivalence between a yes-instance of N3DM and $iRR$.
     We argue as follows: if N3DM has a solution (and let us assume the $m$ triples are indexed $1,2,\ldots, m$), we construct the following draw.
     Each of the three teams corresponding to the $\ell$-th triple $(x_i, y_j, z_k)$ plays against the $\ell$-th dummy team of type (i), (ii), and (iii) respectively, $\ell = 1, \ldots, m$.
     Observe that this constitutes a feasible draw that is fair as the strength of each opponent set equals $B + L + L^2$.

    If the instance of $iRR$ has a solution, it follows that each opponent set has total strength $B+L+L^2$.
     Then, it follows that all opponent sets either have strengths~$(B, L, L^2)$ (variant~1), or~$(x_i, y_j + L, z_k + L^2)$ (variant~2).
    Consider now the $2m$ opponent sets of the teams of type~(ii).
    As these opponent sets are pairwise disjoint (if not, the team in the intersection would play against two teams from type~(ii), but then this opponent set would not be one of the two variants), there are~$m$ disjoint triples of variant~2 among these opponent sets, which corresponds to a solution of N3DM\@.
\end{proof}

We leave it to the reader to verify that one can extend the proof of Theorem~\ref{th:basecase} to hold for any fixed~$k \ge 3$.
Indeed, one can extend the proof with teams of strength~$0$ when~$k > 3$.
\begin{corollary}
    $iRR$ is NP-complete for any fixed $k\ge3$.
\end{corollary}

\subsection{The extended case}\label{sec:extendedcasekgeq3}
Using variations of the proof of Theorem~\ref{th:basecase}, we prove complexity results for extended versions of $iRR$.
\begin{theorem}\label{th:iRR2-3}
    $iRR_{2,3}$ is NP-complete for $k=3$.
\end{theorem}
\begin{proof}
    We build an instance of $iRR_{2,3}$ as follows: let there be $n:=6m$ teams, each of the $p:=3$ pots containing $2m$ teams, and let each team play against $k:=3$ other teams. We set $c:=2$. The $2m$ teams in each pot are divided into $m$ teams from association A and $m$ teams from association C. We set the strength of the $m$ teams from association A in pot 1 (2,3) equal to $\{x_i\}$ ($\{y_j\}$, $\{z_k\}$). We set the strength of all $3m$ teams from association C equal to $\frac{B}3$. This specifies an instance of $iRR_{2,3}$ with $c=2$, and $p=k=3$.

    If N3DM has a solution (and let us assume the $m$ triples are indexed $1,2,\ldots, m$), we construct the following draw. The three teams corresponding to the $\ell$-th triple $(x_i, y_j, z_k)$ play against the $\ell$-th team from association C in each of the three pots, $\ell = 1, \ldots, m$. Observe that this draw ensures that all teams play against three teams that are (i) from another association, come from pots 1, 2 and 3, and that the total strength of each opponent set equals $B$.

    If the instance of $iRR_{2,3}$ has a solution, consider a team from association C in pot 1. This team must play against three teams from association A corresponding to a triple $(x_i, y_j, z_k)$ with total strength $B$. Moreover, observe that any two teams from association C in pot 1 play against two triples of teams that are pairwise disjoint (if not, the team in the intersection would play against two teams from pot 1). It follows that the opponent sets of the teams of association C in pot 1 constitute a feasible solution to N3DM.
\end{proof}

\begin{theorem}\label{th:iRRn-3}
    $iRR_{n,3}$ is NP-complete for $k=3$.
\end{theorem}
\begin{proof}
    We again make a reduction from Numerical 3-Dimensional Matching (N3DM).
    We construct an instance of $iRR_{n,3}$ as follows.
    Let there be~$n \coloneqq 6m$ teams, each team from a different association, with each of the $p=3$ pots containing~$2m$ teams.
    In pot~$1$, $m$ teams are given strengths~$\{x_i\}$, while the remaining~$m$ `dummy' teams are given strength~$B$.
    In pots~$2$ and~$3$, $m$ teams are given strengths~$\{y_i\}$ and~$\{z_i\}$, respectively, while the~$m$ remaining `dummy' teams in each pot have strength~$0$.

    Similarly as before: if N3DM has a solution (and let us assume the $m$ triples are indexed $1,2,\ldots, m$), we construct the following draw. The three teams corresponding to the $\ell$-th triple $(x_i, y_j, z_k)$ play against the $\ell$-th dummy team in each of the three pots, $\ell = 1, \ldots, m$. Observe that this draw ensures that all teams play against three teams that come from pots 1, 2 and 3, and that the total strength of each opponent set equals $B$.

    If the instance of $iRR_{n,3}$ has a solution, we look at the opponent sets of the teams in pot~$2$.
    Observe that the opponent sets either have strengths~$(B, 0, 0)$ (variant~1), or~$(x_i, y_j, z_k)$ (variant~2).
    As these opponent sets are pairwise disjoint (if not, the team in the intersection would play against two teams from pot~$2$), there are~$m$ disjoint triples of variant~2, corresponding to a solution of N3DM\@.
\end{proof}
For $iRR_{2,1}$ we can make a similar type of argument:
\begin{theorem}\label{th:iRR2-1}
    $iRR_{2,1}$ is NP-complete for $k=3$.
\end{theorem}
\begin{proof}
    We again make a reduction from Numerical 3-Dimensional Matching, by constructing an instance of $iRR_{2,1}$ as follows.
    Let there be~$n \coloneqq 6m$ teams, with $3m$ teams from association A, and $3m$ teams from association C.
    From association A, there are 3 groups of $m$ teams with strengths $\{x_i\}$, $\{y_i + L\}$ and $\{z_i + L\}$, respectively, while in association C 3 groups of $m$ teams have strengths $\{B/3\}$, $\{B/3+L\}$ and $\{B/3+L^2\}$.
    The remainder of the argument is analogous to the proofs of Theorems~\ref{th:basecase} and \ref{th:iRR2-3}.
\end{proof}

The proofs of Theorems~\ref{th:iRR2-3}, \ref{th:iRRn-3}, and \ref{th:iRR2-1} can be extended to hold for any fixed~$k \ge 3$, by introducing teams with strength~$0$.
This leads to the following corollary.
\begin{corollary}
    $iRR_{2,k}$, $iRR_{n,k}$, and $iRR_{2,1}$ are NP-complete for any fixed $k\ge3$.
\end{corollary}

\section{Computing draws: an integer program and experiments}\label{sec:compexp}
We describe how to compute a draw that minimizes the bandwidth in Section~\ref{sec:exp:ip}.
In practice, draws for tournaments are obtained using random sampling.
For various settings, described in Section~\ref{sec:exp:setup}, we perform experiments where we compute minimum-bandwidth draws and compare them to draws obtained by sampling draws uniformly at random.
For each draw, we compute its bandwidth and analyze the connectedness of the draw.
We report if the opponent graph is connected, and when it is, we report the diameter of the graph and the average shortest-path distance of all pairs of vertices.
These results are described in Section~\ref{sec:exp:results}.

\subsection{Computing a minimum-bandwidth draw} \label{sec:exp:ip}
In this section, we describe how to find a draw with minimum bandwidth, using integer programming.
Given an instance of $iRR_{c,p}$, where we use~$T, \mathcal P, k$ to denote the set of teams, the set of pots, and the number of opponents for each team, respectively.
As before, let~$n = |T|$ and~$p = |\mathcal P|$.
If the instance has no pots, then~$\mathcal P = \set{T}$.
We can compute a draw with minimum bandwidth using the following integer program (see \cite{Melkonian2024}).
It uses binary decision variables $x_{i,j}$ to denote whether team $j \in T$ is present in the opponent set of team $i \in T$; further we need two continuous real variables $z_\mathrm{max}$ and $z_\mathrm{min}$.
\begin{alignat}{9}
    &&\text{minimize}\quad && z_\mathrm{max} &- z_\mathrm{min}\label{eq1} && \\
    &&\text{subject to}\quad && \sum_{\substack{j \in P \\ j \neq i}} x_{i,j} &= k / p &&  \qquad\forall i \in T,\ P \in \mathcal P, &&\label{eq2} \\
    &&&& x_{i,j} - x_{j,i} &= 0 && \qquad \forall i,j \in T,\ i \neq j, &&\label{eq3} \\
    &&&& z_\mathrm{max} - \frac{\sum_{j \in T} s_j x_{i,j}}{k}  &\geq  0 && \qquad \forall i \in T, && \label{eq4}\\
    &&&& z_\mathrm{min} - \frac{\sum_{j \in T} s_j x_{i,j}}{k}  &\leq  0 && \qquad \forall i \in T, && \label{eq5}\\
    &&&& x_{i,j} &\in \set{0, 1} &&\qquad\forall{i,j \in T},\ i \neq j.\label{eq6}
\end{alignat}
Constraints~\eqref{eq2} ensure that every team plays exactly~$k/p$ teams from each pot.
Constraints~\eqref{eq3} ensure that team~$j$ appears in the opponent set of team~$i$ when $i$ is in the opponent set of~$j$, and vice versa.
Constraints~\eqref{eq4} and~\eqref{eq5} together with the objective function~\eqref{eq1} ensure that $z_\mathrm{max}$ and $z_\mathrm{min}$ will equal the strength of the strongest and weakest opponent set, respectively.

\subsection{Experimental setup} \label{sec:exp:setup}
We generate instances of $iRR_{n,p}$ with the following parameters:
\begin{itemize}
    \item the number of teams $n \in \set{16, 32}$,
    \item the number of opponents $k \in \set{4, 8, 12, 22}$, with $k \le n$,
    \item the number of pots~$p \in \set{1, 2, 4, 8}$, where~$k$ and~$n$ are a multiple of~$p$,
    \item the upper bound on the team strength~$U \in \set{10, 100}$.
\end{itemize}
We call a 4-tuple of the parameters~$(n, k, p, U)$ a class of instances; there are 42~classes.
For each class of instances, we generate~$100$ problem instances with~$n$ teams having an integer strength drawn uniformly at random from the interval~$[0, U]$, except for the classes with~$n=32$ and~$U=100$, where we generate~$10$ problem instances due to computational limits.
In total, this yields~$3320$ problem instances.
For each of these problem instances, we solve the integer program~\eqref{eq1}--\eqref{eq6}.

Next to computing the minimum-bandwidth draw, we also compute random draws by uniformly sampling draws from the feasible region of the integer program~\eqref{eq1}--\eqref{eq6}, using rejection sampling.
To compare the minimum-bandwidth draw computed by the integer program and uniformly-sampled random draws, we perform~$100$ uniform random samples and compute the minimum-bandwidth draw using Gurobi 12 on a compute cluster, with a maximum running time of one hour.
If the model does not yield an optimal solution within this time limit, the best solution found within the limit is reported.
For each sampled draw, we compute the bandwidth and analyze the connectedness of the opponent graph.
Note that when $k \ge n/2$, all draws will be maximally connected, and we do not need the computational experiments for the reported connectedness measures in these cases.
Our implementation and generated instances are available in \citet{zenodo}.

\subsection{Results} \label{sec:exp:results}
The results of the random-draw experiments are displayed in Table~\ref{tab:results:draws}, while the results of the minimum-bandwidth instances are displayed in Table~\ref{tab:results:model}.
Thus, each number in the columns labeled `avg', `min', and `max' is an average of~$10\,000$ numbers, except for the classes with~$n=32$ and~$U=100$ where it is an average of~$1\,000$ numbers.
The tables show for each class of instances~$(n, k, p, U)$ the average, minimum, and maximum bandwidth, diameter, and average distance between any pair of teams, aggregated over all draws belonging to that class.
The rightmost column shows the number of draws where the opponent graph is disconnected.
When the opponent graph is disconnected, the draw is excluded from the diameter and average distance columns.

First, in Table~\ref{tab:results:draws}, each row represents a class consisting of~$100$ instances, where for each instance~$100$ draws are uniformly sampled, leading to an aggregation of~$10\,000$ draws for each instance class.
We observe the following concerning the bandwidth:
\begin{itemize}
    \item the bandwidth decreases as~$k$ increases;
    \item the bandwidth decreases as~$p$ increases; and
    \item the bandwidth increases a bit as~$n$ increases.
\end{itemize}
Regarding the connectivity of the draw, we observe:
\begin{itemize}
    \item almost all draws have a connected draw graph;
    \item the diameter and average distance decrease when $k$ decreases; and
    \item the diameter and average distance increase a bit when $n$ increases;
\end{itemize}

Second, in Table~\ref{tab:results:model} each row represents a class consisting of~$100$ instances, except for the classes with~$U=100$ and~$n=32$, which represent~$10$ instances.
Almost all instances give an optimal solution within this time limit, except for 336 instances with~$n=32$ and~$U=100$.
We observe that the bandwidth increases a bit as~$p$ increases.
Notably, for many instance classes there exists an instance where a bandwidth of~$0$ can be achieved.
Regarding the connectivity of the draw, we observe:
\begin{itemize}
    \item disconnected opponent graphs occur for several instances with~$k=4$, where for the instances with~$U=100$ and~$n=16$ almost all draws are disconnected;
    \item the draws in these classes that are connected have a high diameter close to or exactly~$4$; and
    \item all draws for $n=32$ and $k=p=4$ are disconnected.
\end{itemize}
However, since the model does not optimize for connectedness, it could be the case that other draws with the same bandwidth exist that are connected.
The average bandwidth increases with the number of pots, as pots yield a more constrained model, albeit that the increase in bandwidth is only slight.

When comparing Tables~\ref{tab:results:draws} and~\ref{tab:results:model}, a main conclusion is that the minimum-bandwidth draws have significantly less bandwidth than the uniformly-sampled random draws.
This is confirmed by comparing the maximum bandwidths that occur within the instance classes.
When it comes to connectivity, we see that the minimum-bandwidth draws are slightly less connected than the uniformly-sampled random draws.
This follows from comparing the diameters and the average distances between teams, while noting that disconnected draws occur more frequently when minimizing the bandwidth.

\begin{table}[p]
    \footnotesize
    \caption{Aggregated characteristics of randomly sampled draws for the instance classes in the test set.} \label{tab:results:draws}
    \vspace*{1em}
    \begin{tabular}{rrrrrrrrrrrrrr}
    \toprule
     &&&& \multicolumn{3}{c}{bandwidth} & \multicolumn{3}{c}{diameter} & \multicolumn{3}{c}{average distance} &  \\
    \cmidrule(lr){5-7} \cmidrule(lr){8-10} \cmidrule(lr){11-13}
    $U$ & $n$ & $k$ & $p$ & avg & min & max & avg & min & max & avg & min & max & $\#\text{d}$ \\
    \midrule\addlinespace[2.5pt]
    10 & 16 & 4 & 1   & 4.86 & 1.25 & 8.75 & 3.19 & 3 & 5 & 1.97 & 1.82 & 2.27 &   \\
       &    &   & 2   & 2.58 & 0.75 & 4.50 & 3.28 & 3 & 6 & 1.99 & 1.84 & 2.56 &   \\
       &   &   & 4    & 1.47 & 0.50 & 2.50 & 3.49 & 3 & 5 & 2.03 & 1.87 & 2.37 & 2 \\
       &   & 8 & 1    & 2.77 & 1.00 & 5.50 & 2.00 & 2 & 2 & 1.47 & 1.47 & 1.47 &   \\
       &   &   & 2    & 1.50 & 0.38 & 3.12 & 2.00 & 2 & 2 & 1.47 & 1.47 & 1.47 &   \\
       &   &   & 4    & 0.87 & 0.25 & 1.88 & 2.00 & 2 & 2 & 1.47 & 1.47 & 1.47 &   \\
       &   &   & 8    & 0.50 & 0.00 & 1.12 & 2.00 & 2 & 2 & 1.47 & 1.47 & 1.47 &   \\
       &   & 12 & 1   & 1.57 & 0.50 & 3.00 & 2.00 & 2 & 2 & 1.20 & 1.20 & 1.20 &   \\
       &   &    & 2   & 0.86 & 0.25 & 1.58 & 2.00 & 2 & 2 & 1.20 & 1.20 & 1.20 &   \\
       &   &    & 4   & 0.47 & 0.08 & 0.83 & 2.00 & 2 & 2 & 1.20 & 1.20 & 1.20 &   \\
    \midrule
       & 32 & 4 & 1   & 6.00 & 3.00 & 9.50 & 4.35 & 4 & 6 & 2.55 & 2.43 & 2.79 &   \\
       &    &   & 2   & 3.09 & 1.50 & 4.75 & 4.43 & 4 & 6 & 2.57 & 2.43 & 2.97 &   \\
       &    &   & 4   & 1.70 & 0.75 & 2.50 & 4.58 & 4 & 6 & 2.60 & 2.45 & 3.00 &   \\
       &    & 8 & 1   & 3.94 & 1.88 & 6.75 & 3.00 & 3 & 3 & 1.80 & 1.76 & 1.85 &   \\
       &    &   & 2   & 2.06 & 1.00 & 3.62 & 3.00 & 3 & 3 & 1.81 & 1.77 & 1.87 &   \\
       &    &   & 4   & 1.14 & 0.38 & 2.12 & 3.00 & 3 & 3 & 1.82 & 1.78 & 1.87 &   \\
       &    &   & 8   & 0.68 & 0.25 & 1.25 & 3.00 & 3 & 4 & 1.86 & 1.82 & 1.91 &   \\
       &    & 12 & 1  & 3.01 & 1.42 & 5.33 & 2.13 & 2 & 3 & 1.61 & 1.61 & 1.62 &   \\
       &    &    & 2  & 1.57 & 0.67 & 2.75 & 2.16 & 2 & 3 & 1.61 & 1.61 & 1.62 &   \\
       &    &    & 4  & 0.85 & 0.33 & 1.67 & 2.29 & 2 & 3 & 1.61 & 1.61 & 1.62 &   \\
       &    & 22 & 1  & 1.54 & 0.73 & 2.77 & 2.00 & 2 & 2 & 1.29 & 1.29 & 1.29 &   \\
       &    &    & 2  & 0.81 & 0.36 & 1.41 & 2.00 & 2 & 2 & 1.29 & 1.29 & 1.29 &   \\
    \midrule
    100 & 16 & 4 & 1  & 45.25 & 18.25 & 83.25 & 3.19 & 3 & 5 & 1.97 & 1.83 & 2.33 &   \\
        &    &   & 2  & 24.03 & 6.50 & 40.75 & 3.27 & 3 & 5 & 1.99 & 1.85 & 2.48 &   \\
        &    &   & 4  & 13.67 & 4.00 & 22.75 & 3.49 & 3 & 5 & 2.03 & 1.88 & 2.37 & 3 \\
        &    & 8 & 1  & 26.84 & 9.50 & 56.00 & 2.00 & 2 & 2 & 1.47 & 1.47 & 1.47 &   \\
        &    &   & 2  & 13.68 & 3.75 & 32.75 & 2.00 & 2 & 2 & 1.47 & 1.47 & 1.47 &   \\
        &    &   & 4  & 8.08 & 2.25 & 16.88 & 2.00 & 2 & 2 & 1.47 & 1.47 & 1.47 &   \\
        &    &   & 8  & 4.47 & 0.88 & 9.50 & 2.00 & 2 & 2 & 1.47 & 1.47 & 1.47 &   \\
        &    & 12 & 1 & 14.70 & 4.08 & 28.33 & 2.00 & 2 & 2 & 1.20 & 1.20 & 1.20 &   \\
        &    &    & 2 & 7.84 & 2.67 & 14.00 & 2.00 & 2 & 2 & 1.20 & 1.20 & 1.20 &   \\
        &    &    & 4 & 4.30 & 0.92 & 7.25 & 2.00 & 2 & 2 & 1.20 & 1.20 & 1.20 &   \\
    \midrule
        & 32 & 4 & 1  & 55.24 & 21.75 & 85.50 & 4.35 & 4 & 6 & 2.55 & 2.41 & 2.79 &   \\
        &    &   & 2  & 28.77 & 14.25 & 44.75 & 4.43 & 4 & 6 & 2.57 & 2.43 & 2.86 &   \\
        &    &   & 4  & 15.52 & 8.00 & 24.50 & 4.57 & 4 & 6 & 2.60 & 2.46 & 3.04 &   \\
        &    & 8 & 1  & 36.25 & 15.88 & 65.00 & 3.00 & 3 & 3 & 1.80 & 1.76 & 1.86 &   \\
        &    &   & 2  & 19.05 & 9.00 & 34.25 & 3.00 & 3 & 3 & 1.81 & 1.77 & 1.87 &   \\
        &    &   & 4  & 10.21 & 4.75 & 18.38 & 3.00 & 3 & 3 & 1.82 & 1.77 & 1.88 &   \\
        &    &   & 8  & 5.65 & 2.25 & 10.00 & 3.00 & 3 & 3 & 1.86 & 1.82 & 1.92 &   \\
        &    & 12 & 1 & 27.76 & 13.92 & 52.83 & 2.11 & 2 & 3 & 1.61 & 1.61 & 1.62 &   \\
        &    &    & 2 & 14.30 & 5.58 & 25.58 & 2.16 & 2 & 3 & 1.61 & 1.61 & 1.62 &   \\
        &    &    & 4 & 7.73 & 3.75 & 14.58 & 2.28 & 2 & 3 & 1.61 & 1.61 & 1.62 &   \\
        &    & 22 & 1 & 14.56 & 7.27 & 25.82 & 2.00 & 2 & 2 & 1.29 & 1.29 & 1.29 &   \\
        &    &    & 2 & 7.45 & 3.50 & 14.14 & 2.00 & 2 & 2 & 1.29 & 1.29 & 1.29 &   \\
    \bottomrule
    \end{tabular}
\end{table}

\begin{table}[p]
    \footnotesize
    \caption{Aggregated characteristics of minimum-bandwidth draws for the instance classes in the test set.} \label{tab:results:model}
    \vspace*{1em}
    \begin{tabular}{rrrrrrrrrrrrrr}
    \toprule
    &&&& \multicolumn{3}{c}{bandwidth} & \multicolumn{3}{c}{diameter} & \multicolumn{3}{c}{average distance} &  \\
    \cmidrule(lr){5-7} \cmidrule(lr){8-10} \cmidrule(lr){11-13}
    $U$ & $n$ & $k$ & $p$ & avg & min & max & avg & min & max & avg & min & max & $\#\text{d}$ \\
    \midrule\addlinespace[2.5pt]
    10 & 16 & 4 & 1   & 0.18 & 0.00 & 0.50 & 3.62 & 3 & 4 & 2.10 & 1.93 & 2.33 & 1 \\
       &    &   & 2   & 0.19 & 0.00 & 0.25 & 3.75 & 3 & 5 & 2.11 & 1.94 & 2.49 & 4 \\
       &   &   & 4    & 0.26 & 0.00 & 1.00 & 3.76 & 3 & 4 & 2.15 & 1.96 & 2.33 & 6 \\
       &   & 8 & 1    & 0.06 & 0.00 & 0.12 & 2.00 & 2 & 2 & 1.47 & 1.47 & 1.47 &   \\
       &   &   & 2    & 0.07 & 0.00 & 0.12 & 2.00 & 2 & 2 & 1.47 & 1.47 & 1.47 &   \\
       &   &   & 4    & 0.06 & 0.00 & 0.12 & 2.00 & 2 & 2 & 1.47 & 1.47 & 1.47 &   \\
       &   &   & 8    & 0.06 & 0.00 & 0.25 & 2.00 & 2 & 2 & 1.47 & 1.47 & 1.47 &   \\
       &   & 12 & 1   & 0.06 & 0.00 & 0.25 & 2.00 & 2 & 2 & 1.20 & 1.20 & 1.20 &   \\
       &   &    & 2   & 0.07 & 0.00 & 0.25 & 2.00 & 2 & 2 & 1.20 & 1.20 & 1.20 &   \\
       &   &    & 4   & 0.08 & 0.00 & 0.25 & 2.00 & 2 & 2 & 1.20 & 1.20 & 1.20 &   \\
    \midrule
       & 32 & 4 & 1   & 0.21 & 0.00 & 0.50 & 4.65 & 4 & 6 & 2.62 & 2.50 & 2.91 &   \\
       &    &   & 2   & 0.22 & 0.00 & 0.25 & 4.65 & 4 & 6 & 2.63 & 2.50 & 2.97 &   \\
       &    &   & 4   & 0.23 & 0.00 & 0.50 & 4.90 & 4 & 7 & 2.69 & 2.52 & 3.31 & 1 \\
       &    & 8 & 1   & 0.10 & 0.00 & 0.12 & 3.00 & 3 & 3 & 1.82 & 1.79 & 1.90 &   \\
       &    &   & 2   & 0.09 & 0.00 & 0.12 & 3.00 & 3 & 3 & 1.82 & 1.78 & 1.87 &   \\
       &    &   & 4   & 0.10 & 0.00 & 0.12 & 3.00 & 3 & 3 & 1.83 & 1.79 & 1.87 &   \\
       &    &   & 8   & 0.10 & 0.00 & 0.12 & 3.00 & 3 & 3 & 1.87 & 1.84 & 1.93 &   \\
       &    & 12 & 1  & 0.08 & 0.00 & 0.08 & 2.40 & 2 & 3 & 1.61 & 1.61 & 1.62 &   \\
       &    &    & 2  & 0.08 & 0.00 & 0.08 & 2.28 & 2 & 3 & 1.61 & 1.61 & 1.62 &   \\
       &    &    & 4  & 0.07 & 0.00 & 0.08 & 2.38 & 2 & 3 & 1.61 & 1.61 & 1.62 &   \\
       &    & 22 & 1  & 0.04 & 0.00 & 0.05 & 2.00 & 2 & 2 & 1.29 & 1.29 & 1.29 &   \\
       &    &    & 2  & 0.04 & 0.00 & 0.05 & 2.00 & 2 & 2 & 1.29 & 1.29 & 1.29 &   \\
    \midrule
    100 & 16 & 4 & 1  & 0.38 & 0.00 & 2.50 & 4.00 & 4 & 4 & 2.28 & 2.17 & 2.33 & 93 \\
        &    &   & 2  & 0.58 & 0.00 & 4.00 & 3.90 & 3 & 4 & 2.27 & 2.03 & 2.33 & 90 \\
        &    &   & 4  & 1.25 & 0.00 & 5.00 & 4.00 & 4 & 4 & 2.32 & 2.30 & 2.33 & 92 \\
        &    & 8 & 1  & 0.06 & 0.00 & 0.12 & 2.00 & 2 & 2 & 1.47 & 1.47 & 1.47 &   \\
        &    &   & 2  & 0.06 & 0.00 & 0.12 & 2.00 & 2 & 2 & 1.47 & 1.47 & 1.47 &   \\
        &    &   & 4  & 0.10 & 0.00 & 1.88 & 2.00 & 2 & 2 & 1.47 & 1.47 & 1.47 &   \\
        &    &   & 8  & 0.15 & 0.00 & 2.12 & 2.00 & 2 & 2 & 1.47 & 1.47 & 1.47 &   \\
        &    & 12 & 1 & 0.13 & 0.00 & 0.75 & 2.00 & 2 & 2 & 1.20 & 1.20 & 1.20 &   \\
        &    &    & 2 & 0.17 & 0.00 & 2.33 & 2.00 & 2 & 2 & 1.20 & 1.20 & 1.20 &   \\
        &    &    & 4 & 0.32 & 0.00 & 2.42 & 2.00 & 2 & 2 & 1.20 & 1.20 & 1.20 &   \\
    \midrule
        & 32 & 4 & 1  & 1.23 & 0.50 & 1.50 & 4.70 & 4 & 6 & 2.62 & 2.49 & 2.90 &   \\
        &    &   & 2  & 0.61 & 0.25 & 1.25 & 4.83 & 4 & 6 & 2.70 & 2.54 & 3.01 & 5 \\
        &    &   & 4  & 0.25 & 0.25 & 0.25 & -- & -- & -- & -- & -- & -- & 10 \\
        &    & 8 & 1  & 0.41 & 0.25 & 0.62 & 3.00 & 3 & 3 & 1.81 & 1.78 & 1.83 &   \\
        &    &   & 2  & 0.29 & 0.12 & 0.50 & 3.00 & 3 & 3 & 1.81 & 1.78 & 1.84 &   \\
        &    &   & 4  & 0.21 & 0.12 & 0.38 & 3.00 & 3 & 3 & 1.83 & 1.81 & 1.85 &   \\
        &    &   & 8  & 0.09 & 0.00 & 0.12 & 3.10 & 3 & 4 & 1.94 & 1.85 & 2.18 &   \\
        &    & 12 & 1 & 0.23 & 0.17 & 0.33 & 2.00 & 2 & 2 & 1.61 & 1.61 & 1.61 &   \\
        &    &    & 2 & 0.17 & 0.08 & 0.25 & 2.30 & 2 & 3 & 1.61 & 1.61 & 1.61 &   \\
        &    &    & 4 & 0.10 & 0.08 & 0.17 & 2.40 & 2 & 3 & 1.61 & 1.61 & 1.61 &   \\
        &    & 22 & 1 & 0.11 & 0.08 & 0.12 & 2.00 & 2 & 2 & 1.61 & 1.61 & 1.61 & \\
        &    &    & 2 & 0.14 & 0.09 & 0.23 & 2.00 & 2 & 2 & 1.29 & 1.29 & 1.29 &   \\
    \bottomrule
    \end{tabular}
\end{table}

\section{Conclusion and managerial implications} \label{sec:conclusion}
For a variety of reasons, incomplete Round Robins
(iRR) are becoming more popular in practice. The fact that less matches
need to be played to come to a ranking, the fact that an iRR can be more
 interesting than a set
 of (small) double round robins, and the fact that it is hard to
interpret the standings of an iRR during the season, all lead to the
adoption of iRR as an attractive tournament format. Currently, iRRs
have been implemented from professional leagues (such
 as the Champions League) to youth competitions (such as the Belgian
Youth Hockey competition).

This increasing popularity challenges
tournament organizers to find draws that are as fair as possible.
Indeed, in a setting where teams (or players) have a given
 strength (say an Elo rating) that may vary among the teams, the problem
 of finding opponent sets that are balanced with respect to strength is
far from trivial. Another relevant aspect of the draw is its
connectivity (see Section~\ref{sec:problemdescription}), as a tournament organizer
 would wish to avoid the existence of subsets of teams that play among
themselves. We model the resulting problem as an optimization problem
where we aim to minimize the difference between the strongest opponent
set and the weakest opponent set (the bandwidth),
 and we analyze this problem theoretically. We show experimentally that
randomly generating opponent sets leads to draws with a bandwidth that
is an order of magnitude worse than the minimum possible bandwidth. We
also show experimentally that draws with minimum
 bandwidth are only a bit less connected than random draws.

While random
draws may have the advantage of having ex-ante fairness, we show that
there may exist many draws with a very small bandwidth that lead to a
much more balanced tournament. We expect that these
 findings can support tournament organizers in designing fair
tournaments.

\section*{Funding}
This work was supported by Dutch Research Council (NWO) Gravitation Project NETWORKS, Netherlands, grant number~024.002.003.

\bibliographystyle{abbrvnat}
\bibliography{fair-irr}

\end{document}